\documentclass[reqno]{amsart}
\usepackage{amssymb,amsmath,latexsym,amscd}
\usepackage{hyperref}
\usepackage{enumerate}
\usepackage{tikz-cd}
\usepackage{mathrsfs}
\usepackage{xparse}
\usepackage{mathscinet}
\usepackage{todonotes}
\usepackage{paralist}

\DeclareMathOperator{\im}{im}
\DeclareMathOperator{\op}{op}

\DeclareMathOperator{\s}{span}
\DeclareMathOperator{\ind}{ind}

\renewcommand{\c}{\colon}

\newcommand{\N}{{\mathbb{N}}}

\renewcommand{\phi}{\varphi}
\newcommand{\ve}{\varepsilon}

\newcommand{\pt}{\widehat{\otimes}}
\newcommand{\hk}{\hookrightarrow}
\newcommand{\la}{\langle}
\newcommand{\ra}{\rangle}

\DeclareMathOperator{\Ext}{Ext}
\DeclareMathOperator{\spn}{span}
\newcommand*{\CC}{\mathbb C}
\newcommand*{\DD}{\mathbb D}
\newcommand*{\cB}{\mathscr B}
\newcommand*{\cH}{\mathscr H}
\newcommand*{\lmod}{\mbox{-}\!\mathop{\mathsf{mod}}}
\newcommand*{\rmod}{\mathop{\mathsf{mod}}\!\mbox{-}}
\newcommand*{\bimod}{\mbox{-}\!\mathop{\mathsf{mod}}\!\mbox{-}}
\newcommand*{\CLCS}{\mathsf{CLCS}}
\newcommand*{\Fr}{\mathsf{Fr}}
\newcommand*{\Ban}{\mathsf{Ban}}
\newcommand*{\Vect}{\mathsf{Vect}}
\newcommand*{\alg}{\mathsf{alg}}
\newcommand*{\CBDF}{\mathsf{CBDF}}
\newcommand*{\sC}{\mathsf{C}}

\newcommand*{\Z}{\mathbb Z}
\newcommand*{\dL}{\mathrm L}
\newcommand*{\dR}{\mathrm R}
\newcommand*{\Ptens}{\mathop{\widehat\otimes}}
\newcommand*{\ptens}[1]{\mathop{\widehat\otimes}_{#1}}

\newcommand*{\wt}{\widetilde}
\newcommand*{\ol}{\overline}
\newcommand*{\xra}{\xrightarrow}
\newenvironment{mycompactenum}{\pltopsep=5pt\begin{compactenum}[\upshape (i)]}%
{\end{compactenum}}

\hypersetup{pdfstartview={XYZ null null 1.25}}

\newtheorem{theorem}{Theorem}[section]
\newtheorem{prop}[theorem]{Proposition}
\newtheorem{lemma}[theorem]{Lemma}
\newtheorem{corollary}[theorem]{Corollary}
\theoremstyle{definition}
\newtheorem{definition}[theorem]{Definition}
\newtheorem{example}[theorem]{Example}
\theoremstyle{remark}
\newtheorem{rem}[theorem]{Remark}

\NewDocumentCommand{\tens}{t_}
 {%
  \IfBooleanTF{#1}
   {\tensop}
   {\widehat{\otimes}}%
 }
\NewDocumentCommand{\tensop}{m}
 {%
  \mathbin{\mathop{\widehat{\otimes}}\displaylimits_{#1}}%
 }

\title[Topological amenability and K\"othe co-echelon algebras]{Topological amenability\\
and K\"othe co-echelon algebras}
\author{Alexei Yu. Pirkovskii}
\address{Faculty of Mathematics\\
HSE University\\
6 Usacheva, 119048 Moscow, Russia}
\email{aupirkovskii@hse.ru}

\author{Krzysztof Piszczek}
\address{Faculty of Mathematics and Computer Science\\
A. Mickiewicz University in Pozna{\'n}\\
Umultowska 87\\61-614 Pozna{\'n}\\Poland}
\email{kpk@amu.edu.pl}

\begin{document}
\date{}
\begin{abstract}
We introduce a notion of a topologically flat locally convex module, which extends
the notion of a flat Banach module and which is well adapted to the nonmetrizable setting
(and especially to the setting of DF-modules). By using this notion, we introduce topologically
amenable locally convex algebras and we show that a complete barrelled DF-algebra
is topologically amenable if and only if it is Johnson amenable, extending thereby
Helemskii--Sheinberg's criterion for Banach algebras. As an application, we completely
characterize topologically amenable K\"othe co-echelon algebras.
\end{abstract}
\maketitle

\footnotetext[3]{{\em Acknowledgement.} The research of the first author has been supported
by the Russian Foundation for Basic Research
grant no. 19-01-00447. The research of the second author
has been supported by the National Center of Science, Poland, grant no. UMO-2013/10/A/ST1/00091.}

\section{Introduction}

The paper is devoted to the study of amenability properties in the framework of DF-algebras. These are
algebras with jointly continuous multiplication whose underlying
topological vector spaces are DF-spaces. The category of DF-spaces contains spaces of distributions, e.g. tempered distributions or distributions with compact support. More generally, duals of Fr\'echet spaces belong to this category. In particular, the duals of
K\"othe echelon spaces are DF-spaces. These are the so-called K\"othe co-echelon spaces and this class of objects will be of particular importance for us.

The general study of amenable DF-algebras meets two major difficulties which come from the facts that the category of DF-spaces does not respect subspaces and that there is no Open Mapping Theorem available. This implies that the two well-known approaches to amenability
(namely, Johnson's approach based on derivations \cite{BEJ}
and Helemskii--Sheinberg's approach based on flat modules \cite{Hel_Shein}) which are equivalent in the category of Banach (or Fr\'echet) algebras are potentially inequivalent
in the DF-algebra framework (however, we have no explicit counterexample so far).
The main aim of this paper is to modify the notion of a flat module in such a way that
the above-mentioned problem disappears. The resulting notion of a {\em topologically flat} module
is equivalent to that of a flat module in the case of Banach (or Fr\'echet) modules,
but, in our view, is better adapted to the nonmetrizable setting.
We define {\em topologically amenable algebras} in terms of topologically flat modules,
and we show that topological amenability for complete barrelled DF-algebras
is equivalent to amenability in Johnson's sense. We also obtain a topological amenability
criterion for K\"othe co-echelon algebras, completing thereby recent results
of the second author \cite{KP-amenable,KP-contractible}.

The theory of amenable Banach algebras essentially starts
with the famous result of Johnson \cite[Theorem 2.5]{BEJ} who proved that the convolution algebra $L^1(G)$ is amenable if and only if the locally compact group $G$ is amenable. Since then amenable Banach algebras became an inseparable part of functional analysis and operator algebra theory
(see \cite{Runde_new} for a recent and detailed account).
A few years after the publication of Johnson's memoir, Helemskii and Sheinberg \cite{Hel_Shein}
observed that the notion of an amenable algebra perfectly fits into the general
``Banach homological algebra'' developed earlier by Helemskii \cite{Hel_dim}
(and, independently, by Kiehl and Verdier \cite{KV} and by Taylor \cite{JLT}).
Namely, Helemskii and Sheinberg proved that a Banach algebra $A$ is amenable in Johnson's sense
if and only if the unitization of $A$ is a flat Banach $A$-bimodule.
This result was extended by the first author \cite[Corollary 3.5]{AP-weak} to the setting
of Fr\'echet algebras. In the present article we continue this investigation and study amenability
properties of DF-algebras, with a special emphasis on K\"othe co-echelon algebras.

The paper is organized as follows. The next section is Notation and Preliminaries,
and it contains basic definitions, facts and notation that is used in the sequel.
In Section~\ref{sect:flat_amen}, we introduce and study topologically flat locally convex modules and
topologically amenable locally convex algebras. The main results here are Theorem~\ref{ext-top-flat},
which characterizes topologically flat DF-modules in terms of the $\Ext$ functor, and
Theorem~\ref{topam-der}, which shows that, for complete barrelled DF-algebras, the topological amenability
in our sense is equivalent to the Johnson amenability.
In Section~\ref{sect:coechelon}, we characterize topologically amenable K\"othe co-echelon algebras $k_p(V)$ in terms of the corresponding weight sets $V$
(Theorems~\ref{finite-order} and~\ref{thm:kinf}).
Finally, in Section~\ref{sect:examples} we give some concrete examples of topologically
amenable (and non-amenable) co-echelon algebras. In particular, we construct a
topologically amenable co-echelon algebra of order $\infty$ which, in a sense,
cannot be reduced to a direct sum of $\ell_\infty$ with a contractible co-echelon algebra.

General references are: \cite{MV} for functional analysis, \cite{D,M} for Banach and topological
algebra theory, and \cite{H2} for the homology theory of topological algebras.

\section{Notation and Preliminaries}
\label{sect:not_prelim}

We start by recalling some basic definitions and introducing some notation
that will be used in the sequel.
By a \textit{locally convex algebra} we mean a locally convex space (lcs)
over $\CC$ equipped with a separately continuous associative multiplication.
In general, locally convex algebras are not assumed to have an identity.
Given a locally convex algebra $A$, we denote by
$A_+$ the unconditional unitization of $A$, and we denote by $A^{\op}$ the
opposite algebra, i.e., the lcs $A$ with multiplication $a\cdot b:=ba$.
In what follows, when
using the word ``algebra'' with an adjective that describes a linear topological
property (such as ``complete'', ``Fr\'echet'', ``Banach'', etc.), we mean that
the underlying lcs of the algebra in question has the
specified property. The same applies to locally convex modules (see below).

Given a locally convex algebra $A$, a {\em left locally convex $A$-module}
is an lcs $X$ together with a left $A$-module structure such that the action
$A\times X\to X$ is separately continuous.
Right locally convex
modules and locally convex bimodules are defined similarly.
At some point we will be using a concrete locally convex bimodule $A\otimes\CC$ which is the lcs
$A$ itself with trivial right module action and multiplication as the left module action.

The completed projective tensor product of lcs's $E$ and $F$
will be denoted by $E\Ptens F$, and the completion of $E$ will be denoted by
$\wt{E}$ or by $E^\sim$.
A complete locally convex algebra with jointly continuous multiplication
is called a {\em $\Ptens$-algebra}.
If $A$ is a $\Ptens$-algebra then the assignment $a\otimes b\mapsto ab$ gives rise to the so-called \textit{product map} $\pi_A\c A\Ptens A\to A$.
We will simply write $\pi$ whenever it is clear to which algebra the product map is referred to.
If $A$ is a $\Ptens$-algebra, then a left locally convex $A$-module $X$ is a
{\em left $A$-$\Ptens$-module} if $X$ is complete and if the action of $A$ on $X$ is jointly
continuous. Right $\Ptens$-modules and $\Ptens$-bimodules are defined similarly.
The category of left $A$-$\Ptens$-modules (respectively, of right $A$-$\Ptens$-modules,
of $A$-$B$-$\Ptens$-bimodules) will be denoted by $A\lmod$
(respectively, $\rmod A$, $A\bimod B$).
Note that $A$-$\Ptens$-bimodules are nothing but left unital $A^e$-$\Ptens$-modules, where
$A^e:=A_+\Ptens A_+^{\op}$ is the enveloping algebra of $A$
(see \cite[\S II.5.2]{H2}).
If $X\in\rmod A$ and $Y\in A\lmod$,
then the \emph{$A$-module projective tensor product} of $X$ and $Y$ is defined as
\[
X\ptens{A} Y:=(X\Ptens Y/N)^\sim,
\]
where
\[
N:=\overline{\s}\bigl\{x\cdot a\otimes y-x\otimes a\cdot y : a\in A,\, x\in X,\, y\in Y\bigr\}
\subset X\Ptens Y.
\]

If $X$ and $Y$ are two lcs's, then $L(X,Y)$ stands for the vector space
of continuous linear
operators from $X$ to $Y$. We equip $L(X,Y)$ with the topology of uniform convergence
on bounded sets. As usual, we let $X'=L(X,\CC)$. If $A$ is a locally convex algebra and $X,Y$
are left locally convex $A$-modules, then $_AL(X,Y)$ denotes the vector space of continuous linear
$A$-module maps, i.e., operators $T\in L(X,Y)$ satisfying $T(a\cdot x)=a\cdot Tx$ for all $a\in A$,
$x\in X$. In the case of right $A$-modules, resp. $A$-$B$-bimodules, the vector spaces
$L_A(X,Y)$ and $_AL_B(X,Y)$ are defined analogously.

Suppose that $A$, $B$, $C$ are locally convex algebras, $X$ is a locally convex
$B$-$C$-bimodule, and $Y$ is a locally convex $A$-$C$-bimodule. Then
$L_C(X,Y)$ has a natural $A$-$B$-bimodule structure given by
\[
(a\cdot T)(x)=a\cdot T(x),\quad (T\cdot b)(x)=T(b\cdot x)
\qquad (a\in A,\; b\in B,\; T\in L_C(X,Y)).
\]
If the actions of $A$ on $Y$ and of $B$ on $X$ are hypocontinuous with respect to the families
of bounded subsets of $Y$ and $X$, respectively, then $L_C(X,Y)$ is easily seen to be
a locally convex $A$-$B$-bimodule (cf. \cite[Section 3]{JLT}). In particular, this condition is satisfied
provided that the actions are jointly continuous. In particular, for each $\Ptens$-algebra $A$
and each left (respectively, right) $A$-$\Ptens$-module $X$ the dual space $X'$ is a right
(respectively, left) locally convex $A$-module. Note, however, that the action of $A$ on $X'$
need not be jointly continuous.

Let $\CLCS$ denote the category of complete lcs's and continuous linear
maps. Suppose that $\sC\subset\CLCS$ is a full additive subcategory.
We write $\alg(\sC)$ for the category of all $\Ptens$-algebras whose underlying spaces
are objects of $\sC$. If $A$ is a $\Ptens$-algebra, then we denote by $A\lmod(\sC)$
the full subcategory of $A\lmod$ consisting of those modules whose underlying spaces
are objects of $\sC$. The symbols $\rmod A(\sC)$ and $A\bimod B(\sC)$ are understood
in a similar way.

Following \cite{Pir_wdgnucl}
(cf. also \cite{H2}), we say that $\sC$ is {\em admissible} if the following holds:
\begin{itemize}
\item[$(\sC1)$]
if $E\in\sC$ and $F$ is a locally convex space isomorphic to $E$, then $F\in\sC$;
\item[$(\sC2)$]
if $E\in\sC$ and $E_0\subset E$ is a complemented vector subspace, then $E_0\in\sC$;
\item[$(\sC3)$]
if $E,F\in\sC$, then $E\Ptens F\in\sC$.
\end{itemize}

Most of the categories of complete lcs's used in functional analysis are
admissible. In this paper, the concrete admissible subcategories we are mostly interested in are
$\CLCS$ itself, the category $\Ban$ of Banach spaces, the category $\Fr$ of Fr\'echet
spaces, and the category $\CBDF$ of complete barrelled (DF)-spaces.
The admissibility of $\Ban$ and $\Fr$ is well known. As for $\CBDF$, property $(\sC2)$
follows from the fact that the classes of barrelled spaces and of (DF)-spaces are stable
under taking quotients modulo closed subspaces \cite[27.1.(4) and 29.5.(1)]{K1},
while property $(\sC3)$ follows from \cite[41.4.(7) and 41.4.(8)]{K2}.

Let $A$ be a $\Ptens$-algebra, and let $\sC$ be an admissible subcategory of $\CLCS$.
A sequence
\begin{equation}
\label{short_seq}
0 \to X \xra{i} Y \xra{p} Z \to 0
\end{equation}
in $A\lmod(\sC)$ is {\em admissible} if it is split exact in $\CLCS$, i.e., if it has a contracting
homotopy consisting of continuous linear maps. Geometrically, this means that $i$ is a topological
embedding, $p$ is open (i.e., is a quotient map), $i(X)=\ker p$, and $i(X)$ is a complemented
subspace of $Y$.
We say that a morphism $i\colon X\to Y$ (respectively, $p\colon Y\to Z$) in $A\lmod(\sC)$
is an {\em admissible monomorphism} (respectively, an {\em admissible epimorphism})
if it fits into an admissible sequence \eqref{short_seq}.

It is easy to show that $A\lmod(\sC)$ together with the class of all
admissible sequences is an exact category in Quillen's sense \cite{Quillen}.
Therefore most of the main notions and constructions of homological algebra
(projective objects, projective resolutions, derived functors, etc.)
make sense in $A\lmod(\sC)$. For details, we refer to \cite{H2}.
An important property of $A\lmod(\sC)$ is that, if $A\in\alg(\sC)$,
then $A\lmod(\sC)$ has enough projectives. As a consequence,
each covariant functor $F\colon A\lmod(\sC)\to\Vect$
(where $\Vect$ is the category of vector spaces)
has left derived functors
$\dL_n F$, and each contravariant functor $F\colon A\lmod(\sC)\to\Vect$ has right derived
functors $\dR^n F$ ($n\ge 0$). In particular, for each left locally convex $A$-module $Y$
the functor $\Ext^n_A(-,Y)$ is defined to be the $n$th right derived functor of
${_A}L(-,Y)\colon A\lmod(\sC)\to\Vect$. We would like to stress that, in contrast to \cite{H2},
we do not require $Y$ to be an object of $A\lmod(\sC)$. In particular,
we may let $Y=Z'$ for some $Z\in\rmod A(\sC)$. This special case will be essential
in our characterization of topologically flat modules (see Theorem \ref{ext-top-flat}).
In fact, this is the only reason why we have to consider general locally convex modules
rather than $\Ptens$-modules only.

Note that the above facts on $A\lmod(\sC)$ have obvious analogs for $\rmod A(\sC)$
and $A\bimod B(\sC)$. For details, see \cite{H2}.

Let us now recall some facts on strictly exact sequences of locally convex spaces.
Let $\sC$ be an additive category. Following \cite{Schndrs}, we say that a short sequence
\eqref{short_seq}
in $\sC$ is {\em strictly exact} if $i$ is a kernel of $p$ and $p$ is a cokernel of $i$.

\begin{example}
If $\sC=\Vect$, then \eqref{short_seq} is strictly exact iff it is exact in the usual sense.
\end{example}

\begin{example}
\label{example:str_ex_LCS}
If $\sC=\CLCS$, then \eqref{short_seq} is strictly exact iff $i$ is topologically injective,
$i(X)=\ker p$, $p$ is an open map of $Y$ onto $p(Y)$, and $p(Y)$ is dense in $Z$.
This follows from \cite[Proposition 4.1.8]{Prosm_derFA}. Essentially, this means that
$X$ can be identified with a closed subspace of $Y$, and $Z$ is the completion of $Y/X$.
\end{example}

\begin{example}
\label{example:str_ex_Fr}
If $\sC=\Fr$ or $\sC=\Ban$, then \eqref{short_seq} is strictly exact in $\sC$
iff it is strictly exact in $\CLCS$ iff
it is exact (or, equivalently, strictly exact) in $\Vect$. This is essentially a combination of Example~\ref{example:str_ex_LCS}
with the Open Mapping Theorem. See also \cite[Chapter 2]{Wengen}.
\end{example}

The following result is a special case of V.~P.~Palamodov's theorem \cite[Proposition 4.2]{VPP}
(see also \cite[Theorem 2.2.2]{Wengen}). Given a set $S$, let $\ell_\infty(S)$ denote the
Banach space of bounded $\CC$-valued functions on $S$.

\begin{theorem}[Palamodov]
\label{top-exact}
A short sequence \eqref{short_seq} in $\CLCS$ is strictly exact if and only if, for each set $S$,
the sequence
\[
0 \to L(Z,\ell_\infty(S))\to L(Y,\ell_\infty(S)) \to L(X,\ell_\infty(S)) \to 0
\]
is exact in $\Vect$.
\end{theorem}

We end this section with a definition and a collection of basic facts concerning K\"othe co-echelon spaces and algebras. Let $I$ be a countable set, and let
$V:=(v_n)_{n\in\N}$ be a sequence of weights $v_n\c I\to(0,\infty]$ such that
\begin{gather}
\forall\,i\in I\quad\exists\,n\in\N\quad v_n(i)<\infty,\tag{W1}\\
\forall\,n\in\N\quad\forall i\in I\quad v_{n+1}(i)\le v_n(i).\tag{W2}
\end{gather}
For $1\le p\le\infty$ or $p=0$
we define the \textit{K\"othe co-echelon space of order $p$} as
\begin{gather*}
k_p(I,V):=\Big\{x=(x_i)\in\CC^I : \sum_{i\in I}|x_i|^pv_n(i)^p<\infty\text{\,\,for some\,\,}n\in\N\Big\}
\quad (1\le p<\infty),\\
k_{\infty}(I,V):=\Big\{x=(x_i)\in\CC^I : \sup_{i\in I}|x_i|v_n(i)<\infty\text{\,\,for some\,\,}n\in\N\Big\},\\
k_0(I,V):=\Big\{x=(x_i)\in\CC^I : \lim_{i\to\infty}|x_i|v_n(i)=0\text{\,\,for some\,\,}n\in\N\Big\}.
\end{gather*}
We often write $k_p(V)$ for $k_p(I,V)$ when the index set $I$ is clear from the context.
In most examples we actually have $I=\N$ (see Examples \ref{example:phi}--\ref{example:germs}),
but sometimes it is more convenient to let $I=\N\times\N$ (see Example~\ref{example:NN}).

The above definition is a bit unusual since we allow $v_n(i)=\infty$ for some $n\in\N$ and $i\in I$.
However, this less restrictive approach does not affect our proofs and allows us to consider in particular the space $\phi:=\CC^{(\N)}$ of finitely supported sequences
(see Example \ref{example:phi} below).
The space $k_p(I,V)$ is canonically endowed with the inductive limit topology of the system $(\ell_p(I,v_n))_{n\in\N}$ (for $p\ge 1$) or $(c_0(I,v_n))_{n\in\N}$ (for $p=0$),
where $\ell_p(I,v_n)$ and $c_0(I,v_n)$
are the weighted Banach spaces of scalar sequences equipped with their canonical norms. Clearly, if $v_n(i)=\infty$, then $x\in\ell_p(I,v_n)$ implies that $x_i=0$.
Thus we usually write
\[k_p(I,V)=\ind_n\ell_p(I,v_n)\hspace{10pt}(1\le p\le\infty),\qquad k_0(I,V)=\ind_nc_0(I,v_n).\]
Since K\"othe co-echelon spaces are countable inductive limits of Banach spaces,
they are barrelled DF-spaces (see \cite[12.4, Theorem 8]{J}).
By \cite[Theorem 2.3]{BMS}, $k_p(V)$ is complete for all $1\le p\le \infty$.
On the other hand, $k_0(V)$ is not always complete, see \cite[\S31.6]{K1}
or \cite[Theorem 3.7 and Examples 3.11, 4.11.2, 4.11.3]{BMS}.

In many concrete cases (see examples below), K\"othe co-echelon spaces are algebras
with respect to the coordinatewise multiplication of sequences.
A systematic study of such algebras was initiated in \cite{BonDom}.
Recall from \cite[Proposition 2.1]{BonDom} that $k_p(V)$ is an algebra if and only if
\begin{equation}
\forall\,n\in\N\quad\exists\,m\in\N\quad v_m/v_n^2\in\ell_{\infty}\tag{W3}
\end{equation}
(we let $\infty/\infty=1$ for convenience).
Moreover, if (W3) holds, then the multiplication on $k_p(V)$
is automatically jointly continuous [loc. cit.]. From now on, when we write something like
``let $k_p(V)$ be a {\em K\"othe co-echelon algebra}'', we tacitly assume that $V$ is a sequence of weights
satisfying conditions (W1)--(W3),
and that $k_p(V)$ is considered as a locally convex algebra under the coordinatewise
multiplication.

\begin{example}
\label{example:phi}
For each $n\in\N$, define $v_n\colon\N\to (0,\infty]$ by $v_n(j)=1$ for $j\le n$,
and $v_n(j)=\infty$ for $j>n$. Conditions (W1)--(W3) are clearly satisfied,
and $k_p(V)$ is nothing but the algebra $\varphi$ of finite sequences equipped with
the strongest locally convex topology.
\end{example}

\begin{example}
\label{example:dual_power}
Let $R\in [0,+\infty)$, and let $\alpha=(\alpha_i)_{i\in\N}$ be a sequence
of positive numbers increasing to infinity. Consider
the {\em dual power series spaces}\footnote[1]{By \cite[Theorem 2.7]{BMS},
$D\Lambda_R^p(\alpha)$ is topologically isomorphic to the strong dual of the
power series space $\Lambda_{1/R}^q(\alpha)$, where $1/p+1/q=1$.}
\begin{align*}
D\Lambda_R^p(\alpha)&=\Bigl\{ x=(x_j)\in\CC^\N : \sum_j |x_j|^p r^{\alpha_j p}<\infty
\;\text{for some}\; r>R\Bigr\} \quad (1\le p<\infty);\\
D\Lambda_R^\infty(\alpha)&=\Bigl\{ x=(x_j)\in\CC^\N : \sup_j |x_j| r^{\alpha_j}<\infty
\;\text{for some}\; r>R\Bigr\}.
\end{align*}
If $(r_n)$ is a fixed sequence of positive numbers strictly decreasing to $R$, then
we clearly have $D\Lambda_R^p(\alpha)=k_p(V)$, where $v_n(j)=r_n^{\alpha_j}$
for all $n,j\in\N$. We could also consider the space $D\Lambda_R^0(\alpha)=k_0(V)$
with $V$ as above, but the condition that $\alpha_j\to\infty$ easily implies that
$D\Lambda_R^0(\alpha)=D\Lambda_R^\infty(\alpha)$.

An elementary computation shows that $D\Lambda_R^p(\alpha)$ satisfies (W3)
if and only if for each $r>R$ there exists $\rho>R$ such that $\rho\le r^2$.
Equivalently, this means that if $r>R$, then $r^2>R$.
If $R\ge 1$ or $R=0$, then this condition is clearly satisfied,
so $D\Lambda_R^p(\alpha)$ is a K\"othe co-echelon algebra in this case.
If $0<R<1$, then the above condition fails (take any $r\in (R,\sqrt{R}]$).
\end{example}

\begin{example}
\label{example:s'}
Letting $\alpha_j=\log j$ in Example~\ref{example:dual_power}, we see that
$D\Lambda_0^p(\alpha)$ is nothing but the algebra $s'$ of sequences of polynomial
growth.
\end{example}

\begin{example}
\label{example:germs}
If $\alpha_j=j$, then $D\Lambda_R^p(\alpha)$ is topologically isomorphic to the space
of germs of holomorphic functions on the closed disc
$\ol{\DD}_R=\{ z\in\CC : |z| \le R\}$. If $R\ge 1$ or $R=0$, then
the multiplication on $D\Lambda_R^p(\alpha)$
corresponds to the ``componentwise'' multiplication of the Taylor expansions
of holomorphic functions (the {\em Hadamard multiplication}, cf. \cite{ReSa}).
The resulting locally convex algebra will be denoted by $\cH(\ol{\DD}_R)$.
\end{example}

Given $p\in [1,\infty]$ and a sequence $V=(v_n)$ of weights satisfying (W1)--(W3), we say that
$V$ is {\em eventually in $\ell_p$} if $v_n\in\ell_p(I)$ for some $n\in\N$.
Because of (W2), this means precisely that there exists $n\in\N$ such that
$v_k\in\ell_p(I)$ for all $k\ge n$. If $V$ is eventually in $\ell_\infty$, then we say that
$V$ is {\em eventually bounded}. By \cite[Proposition 2.5]{BonDom},
if $1\le p<\infty$, then
\[
\begin{split}
\text{$V$ is eventually in $\ell_p$}
&\iff \text{$V$ is eventually in $\ell_1$}
\iff \text{$k_p(V)$ is unital}\\
&\iff \text{$V$ is eventually bounded, and $k_p(V)$ is nuclear.}
\end{split}
\]
In fact, if the above conditions are satisfied, then we have $k_p(V)=k_q(V)$
for all $p,q\in [1,\infty]\cup\{ 0\}$ (see \cite[Proposition 15]{Bier}).

A comprehensive study of K\"othe co-echelon spaces may be found in \cite{BMS}.
K\"{o}the co-echelon algebras appear as a main object of investigation in \cite{BonDom}
and \cite{KP-amenable,KP-contractible}.

\section{Topological Flatness and Topological Amenability}
\label{sect:flat_amen}

Let $\sC$ be an admissible subcategory of $\CLCS$, and let $A\in\alg(\sC)$.

\begin{definition}
\label{def:topflat}
We say that a module $X\in A\lmod(\sC)$ is {\em topologically flat}
(relative to $\sC$)
if for each short admissible sequence
\begin{equation}
\label{shortadm}
0 \to Y_1\to Y_2 \to Y_3 \to 0
\end{equation}
in $\rmod A(\sC)$ the sequence
\begin{equation}
\label{short_tens}
0 \to Y_1\ptens{A} X \to Y_2\ptens{A} X \to Y_3\ptens{A} X \to 0
\end{equation}
is strictly exact in $\CLCS$.
A right module in $\rmod A(\sC)$ (respectively, a bimodule in
$A\bimod A(\sC)$) is topologically flat if it is topologically flat
as a left module over $A^{\op}$ (respectively, over $A^e$).
\end{definition}

\begin{rem}
\label{flat-top-flat}
According to \cite{H2}, a module $X\in A\lmod(\sC)$ is {\em flat}
(relative to $\sC$)
if for each short admissible sequence \eqref{shortadm}
in $\rmod A(\sC)$ the sequence \eqref{short_tens}
is exact in $\Vect$.
If $\sC\subset\Fr$, then flatness and topological flatness are equivalent
(see Example~\ref{example:str_ex_Fr}).
We conjecture that, in the general case, neither topological flatness implies flatness, nor vice versa. However, we do not have concrete counterexamples at the moment.
\end{rem}

\begin{example}
\label{example:proj_topflat}
Each projective module $P\in A\lmod(\sC)$ is topologically flat.
Indeed, if $P$ is free, i.e., if $P$ is isomorphic to $A_+\Ptens E$ for some $E\in\sC$,
then \eqref{short_tens} is isomorphic to the sequence
\[
0 \to Y_1\Ptens E \to Y_2\Ptens E \to Y_3\Ptens E \to 0,
\]
which is split exact and is {\em a fortiori} strictly exact in $\CLCS$.
Since each projective module is a retract of a free module \cite[III.1.27]{H2}, the result follows.
\end{example}

\begin{prop}
\label{prop:flat_topinj}
A module $X\in A\lmod(\sC)$ is topologically flat if and only if
for each admissible monomorphism $Y\to Z$ in $\rmod A(\sC)$
the induced map $Y\ptens{A} X\to Z\ptens{A} X$ is topologically
injective.
\end{prop}
\begin{proof}
This is immediate from Definition~\ref{def:topflat} and from the fact
that the functor $(-)\ptens{A} X\colon\rmod A\to\CLCS$ preserves cokernels
\cite[Proposition 3.3]{AP}.
\end{proof}

\begin{rem}
For $\sC=\Ban$, Proposition~\ref{prop:flat_topinj} is well known
(cf. \cite[Theorem VII.1.42]{H3}). For $\sC=\Fr$, this fact was observed
in \cite{AP-weak}.
\end{rem}

The following ``adjoint associativity'' (or ``exponential law'') for
locally convex spaces is a kind of folklore. Since we have not found an exact reference,
we give a proof here for the convenience of the reader.

\begin{prop}
\label{prop:adjass_lcs}
Let $X$, $Y$, $Z$ be locally convex spaces. Suppose that $Z$ is complete.
There is a natural linear map
\begin{equation}
\label{adjass}
L(X\Ptens Y,Z)\to L(X,L(Y,Z)), \qquad
f\mapsto (x\mapsto (y\mapsto f(x\otimes y))).
\end{equation}
The above map is a vector space isomorphism in either of the following cases:
\begin{mycompactenum}
\item
$X$ and $Y$ are Fr\'echet spaces;
\item
$X$ and $Y$ are DF-spaces, and $Y$ is barrelled.
\end{mycompactenum}
\end{prop}
\begin{proof}
By the universal property of the projective tensor product
(see, e.g., \cite[41.3.(1)]{K2}),
$L(X\Ptens Y,Z)$ is naturally identified with the space of jointly
continuous bilinear maps from $X\times Y$ to $Z$.
On the other hand, each $\varphi\in L(X,L(Y,Z))$ determines
a separately continuous bilinear map $\varPhi\colon X\times Y\to Z$ via
$\varPhi(x,y)=\varphi(x)(y)$ ($x\in X$, $y\in Y$).
Moreover, the rule $\varphi\mapsto\varPhi$ determines a vector space
isomorphism between $L(X,L(Y,Z))$ and the space of those separately
continuous bilinear maps $X\times Y\to Z$ which are $\cB_Y$-hypocontinuous,
where $\cB_Y$ is the family of all bounded subsets of $Y$
\cite[40.1.(3)]{K2}.
This implies that \eqref{adjass} indeed takes $L(X\Ptens Y,Z)$
to $L(X,L(Y,Z))$, is always injective, and is surjective if and only if
each separately continuous, $\cB_Y$-hypocontinuous bilinear map
from $X\times Y$ to $Z$ is jointly continuous. In case (i), this condition
is clearly satisfied because the separate continuity and the joint continuity
are equivalent for maps $X\times Y\to Z$
(see, e.g., \cite[40.2.(1)]{K2}).
Assume now that (ii) holds, and let $\varPhi\colon X\times Y\to Z$
be a separately continuous, $\cB_Y$-hypocontinuous bilinear map.
Since $Y$ is barrelled, $\varPhi$ is also $\cB_X$-hypocontinuous
by \cite[40.2.(3)]{K2}. Finally, since $X$ and $Y$ are DF-spaces,
and since $\varPhi$ is $(\cB_X,\cB_Y)$-hypocontinuous, we conclude that
$\varPhi$ is jointly continuous \cite[40.2.(10)]{K2}.
In view of the above remarks, this completes the proof.
\end{proof}

\begin{corollary}
\label{hc-jc}
Let $X,Y$ be either Fr\'echet spaces or barrelled DF-spaces. Then there exist natural vector space isomorphisms
\[
(X\Ptens Y)'\cong L(X,Y') \cong L(Y,X').
\]
\end{corollary}

The following is a natural extension of \cite[II.5.22]{H2} to the
locally convex setting.

\begin{prop}
\label{prop:adjass_mod}
Let $A$, $B$, $C$ be $\Ptens$-algebras, and let $X\in A\bimod B$,
$Y\in B\bimod C$, and $Z\in A\bimod C$. There is a natural linear map
\begin{equation}
\label{adjass_mod}
{_A}L_C(X\ptens{B} Y,Z)\to {_A}L_B(X,L_C(Y,Z)), \qquad
f\mapsto (x\mapsto (y\mapsto f(x\otimes y))).
\end{equation}
The above map is a vector space isomorphism if either of the conditions
{\upshape (i), (ii)} of Proposition~{\upshape\ref{prop:adjass_lcs}} are satisfied.
\end{prop}
\begin{proof}
By the universal property of $\ptens{B}$ \cite[II.4.2]{H2},
${_A}L_C(X\ptens{B} Y,Z)$ is naturally identified with the space of those jointly
continuous bilinear maps from $X\times Y$ to $Z$ which are
(1) $B$-balanced, (2) $A$-linear in the first variable, and (3) $C$-linear in the
second variable. A routine calculation shows that a jointly continuous
bilinear map $X\times Y\to Z$ has the above three properties
if and only if the respective linear map $X\to L(Y,Z)$
takes $X$ to $L_C(Y,Z)$ and is an $A$-$B$-bimodule morphism.
The rest follows from Proposition~\ref{prop:adjass_lcs}.
\end{proof}

\begin{corollary}
\label{hc-jc2}
Let $B$ be a $\Ptens$-algebra, $X\in\rmod B$, and $Y\in B\lmod$.
If $X$ and $Y$ are either Fr\'echet spaces or barrelled DF-spaces,
then there exist natural vector space isomorphisms
\[
(X\ptens{B} Y)'\cong L_B(X,Y')\cong {_B}L(Y,X').
\]
\end{corollary}

\begin{corollary}
\label{tensor-hc-jc}
Let $\sC\in\{\Fr,\CBDF\}$,
let $A$ be a $\Ptens$-algebra, and let $X\in \rmod A(\sC)$,
$Y\in A\lmod(\sC)$, $Z\in\sC$. Then there exists a natural vector space
isomorphism
\[
L(X\ptens{A} Y,Z')\cong L_A(Z\Ptens X,Y').
\]
\end{corollary}
\begin{proof}
This follows from Corollaries \ref{hc-jc}, \ref{hc-jc2},
the commutativity of $\Ptens$,
and the associativity of $\tens_A$, since we have
\[
L(X\ptens{A} Y,Z') \cong ((X\ptens{A} Y)\Ptens Z)' \cong ((Z\Ptens X)\ptens{A}Y)'
\cong L_A(Z\Ptens X,Y'). \qedhere
\]
\end{proof}

The following result was proved in \cite[Proposition 3.3]{AP-weak}
for $\sC=\Fr$.
We now give a shorter proof which holds both for $\sC=\Fr$ and
$\sC=\CBDF$.

\begin{prop}
\label{prop:ext-dual}
Let $\sC\in\{\Fr,\CBDF\}$, and
let $A\in\alg(\sC)$.
Then for all $X\in A\lmod(\sC)$, $Y\in\rmod A(\sC)$, $n\in\Z_+$
there is a natural vector space isomorphism $\Ext^n_A(X,Y')\cong\Ext^n_A(Y,X')$.
\end{prop}
\begin{proof}
Let $P_\bullet\to A_+$ be a projective resolution of $A_+$ in $A\bimod A(\sC)$.
Then $P_\bullet\ptens{A} X\to X$ is a projective resolution of $X$ in $A\lmod(\sC)$,
and $Y\ptens{A} P_\bullet\to Y$ is a projective resolution of $Y$ in $\rmod A(\sC)$.
Applying Corollary~\ref{hc-jc2} twice, we obtain natural vector space isomorphisms
\[
\begin{split}
\Ext^n_A(X,Y')
=H^n({_A}L\bigl(P_\bullet\ptens{A} X,Y')\bigr)
&\cong H^n\bigl((Y\ptens{A} P_\bullet \ptens{A} X)'\bigr)\\
&\cong H^n\bigl(L_A(Y\ptens{A} P_\bullet,X')\bigr)
=\Ext^n_A(Y,X'). \qedhere
\end{split}
\]
\end{proof}

Our next theorem generalizes \cite[Proposition 3.4]{AP-weak}.

\begin{theorem}
\label{ext-top-flat}
Let $\sC\in\{\Fr,\CBDF\}$, and
let $A\in\alg(\sC)$.
The following properties of $X\in A\lmod(\sC)$ are equivalent:
\begin{mycompactenum}
\item $X$ is topologically flat;
\item $\Ext_A^1(Y,X')=0\quad\forall\,\,Y\in \rmod A(\sC)$;
\item
$\Ext_A^1(X,Y')=0\quad\forall\,\,Y\in \rmod A(\sC)$;
\item
$\Ext_A^n(Y,X')=0\quad\forall\,\,Y\in \rmod A(\sC),\quad\forall\,\,n\in\N$;
\item
$\Ext_A^n(X,Y')=0\quad\forall\,\,Y\in \rmod A(\sC),\quad\forall\,\,n\in\N$;
\item
the functor $L_A(-,X')\colon\rmod A(\sC)\to\Vect$ takes
short admissible sequences to exact sequences.
\end{mycompactenum}
\end{theorem}
\begin{proof}
$\mathrm{(ii)\Leftrightarrow (iii)}$, $\mathrm{(iv)\Leftrightarrow (v)}$:
these are special cases
of Proposition~\ref{prop:ext-dual}.

$\mathrm{(ii)\Leftrightarrow (iv)\Leftrightarrow (vi)}$: these are special cases of
\cite[III.3.7]{H2}.

$\mathrm{(i)\Rightarrow (vi)}$.
By assumption,
for each short admissible sequence \eqref{shortadm}
in $\rmod A(\sC)$ the sequence \eqref{short_tens}
is strictly exact in $\CLCS$. By Palamodov's Theorem~\ref{top-exact}, the
dual sequence
\begin{equation}
\label{dual_tens}
0 \to (Y_3\ptens{A} X)' \to (Y_2\ptens{A} X)' \to (Y_1\ptens{A} X)' \to 0
\end{equation}
is exact in $\Vect$. Corollary~\ref{hc-jc2} implies that \eqref{dual_tens} is isomorphic to
\begin{equation}
\label{Hom_X'}
0 \to L_A(Y_3,X') \to L_A(Y_2,X')\to L_A(Y_1,X') \to 0.
\end{equation}
This yields (vi).

$\mathrm{(vi)\Rightarrow (i)}$.
We want to show that
for each short admissible sequence \eqref{shortadm}
in $\rmod A(\sC)$ the sequence \eqref{short_tens}
is strictly exact in $\CLCS$. By Palamodov's Theorem~\ref{top-exact},
this means precisely that for each set $S$ the sequence
\begin{equation}
\label{L_tens_linf}
0 \to L(Y_3\ptens{A} X,\ell_\infty(S)) \to L(Y_2\ptens{A} X,\ell_\infty(S))
\to L(Y_1\ptens{A} X,\ell_\infty(S)) \to 0
\end{equation}
is exact in $\Vect$. Taking into account the isomorphism
$\ell_\infty(S)\cong (\ell_1(S))'$ and applying Corollary~\ref{tensor-hc-jc},
we see that \eqref{L_tens_linf} is isomorphic to
\begin{equation}
\label{Homl1_X'}
0 \to L_A(\ell_1(S)\Ptens Y_3,X') \to L_A(\ell_1(S)\Ptens Y_2,X')
\to L_A(\ell_1(S)\Ptens Y_1,X') \to 0.
\end{equation}
Since $\ell_1(S)\Ptens Y_\bullet$ is admissible in $\rmod A(\sC)$,
we see that \eqref{Homl1_X'}
is exact in $\Vect$ by (vi). In view of the above remarks, this completes the proof.
\end{proof}

The next proposition shows that a flat Banach module over a Banach algebra
remains topologically flat if we consider it as an object of the bigger category
of Fr\'echet modules or of complete barrelled DF-modules.

\begin{prop}
\label{banach-top-flat}
Let $A$ be a Banach algebra and let $X$ be a left Banach $A$-module.
The following conditions are equivalent:
\begin{mycompactenum}
\item
$X$ is flat (or, equivalently, topologically flat) relative to $\Ban$;
\item
$X$ is flat (or, equivalently, topologically flat) relative to $\Fr$;
\item
$X$ is topologically flat relative to $\CBDF$.
\end{mycompactenum}
\end{prop}
\begin{proof}
Clearly, each of the conditions (ii) and (iii) implies (i). Conversely, let
$\sC$ denote either of the categories $\Fr$ or $\CBDF$, and suppose that (i) holds.
By \cite[VII.1.14]{H2}, condition (i) means precisely
that $X'$ is injective in $\rmod A(\Ban)$.
Using \cite[III.1.31]{H2}, we see that $X'$ is a retract of
$L(A_+,X')$ in $\rmod A(\Ban)$. Hence for each short admissible sequence
$Y_\bullet$ in $\rmod A$ the sequence $L_A(Y_\bullet,X')$ is a retract
of $L_A(Y_\bullet,L(A_+,X'))$. On the other hand,
\cite[Proposition 3.2]{JLT} implies that
\[
L_A(Y_\bullet,L(A_+,X')) \cong L(Y_\bullet,X').
\]
Hence $L_A(Y_\bullet,X')$ is a retract of $L(Y_\bullet,X')$, which is clearly
exact in $\Vect$. Therefore $L_A(Y_\bullet,X')$ is exact in $\Vect$.
Applying Theorem \ref{top-exact}, we conclude that
$X$ is topologically flat in $A\lmod(\sC)$.
\end{proof}

\begin{rem}
The equivalence of (i) and (ii) in Proposition~\ref{banach-top-flat}
was proved in \cite[Proposition 4.11]{AP}.
\end{rem}

We now turn to topological amenability, using Helemskii--Sheinberg's approach
\cite{Hel_Shein} as a motivation.
Let $\sC$ be an admissible subcategory of $\CLCS$, and let $A\in\alg(\sC)$.

\begin{definition}
We say that $A$ is \emph{topologically amenable} (relative to $\sC$)
if $A_+$ is topologically flat in $A\bimod A(\sC)$.
\end{definition}

\begin{rem}
According to \cite{H2}, $A$ is {\em amenable} if $A_+$ is flat
in $A\bimod A(\sC)$. As in Remark~\ref{flat-top-flat}, we would like
to stress that amenability and topological amenability are
formally different in the general case,
but they are equivalent if $\sC\subset\Fr$.
\end{rem}

\begin{example}
\label{example:contr_topamen}
Recall from \cite[Chap. VII]{H3} (see also \cite[Postscript]{H2})
that $A$ is {\em contractible} if $A_+$ is projective
in $A\bimod A(\sC)$. Since projective modules are topologically flat
(see Example~\ref{example:proj_topflat}), we conclude that each
contractible algebra is topologically amenable.
\end{example}

Recall that the amenability of a Banach algebra can be rephrased
in the language of derivations. Our next result gives a similar characterization
in the categories $\Fr$ and $\CBDF$. For Fr\'echet algebras, this
was proved in \cite[Corollary 3.5]{AP-weak}.

\begin{theorem}
\label{topam-der}
Let $\sC\in\{\Fr,\CBDF\}$, and
let $A\in\alg(\sC)$.
Then $A$ is topologically amenable relative to $\sC$
if and only if for each $X\in A\bimod A(\sC)$ every continuous
derivation $A\to X'$ is inner.
\end{theorem}
\begin{proof}
It is a standard fact (see, e.g., \cite[Chap. I, Subsection 2.1]{H2})
that every continuous derivation $A\to X'$ is inner if and only if
$\cH^1(A,X')=0$, where $\cH^1(A,X')$ is the 1st continuous Hochschild
cohomology group of $A$ with coefficients in $X'$.
By \cite[III.4.9]{H2}, we have a vector space isomorphism
$\cH^1(A,X')\cong\Ext^1_{A^e}(A_+,X')$.
Now the result follows from Theorem~\ref{ext-top-flat}.
\end{proof}

In the $\CBDF$ category it is also possible to relate topological amenability to amenability.

\begin{corollary}
\label{am-top-am}
Let $A$ be a complete barrelled DF-algebra which is amenable relative to $\CBDF$.
Then $A$ is topologically amenable relative to $\CBDF$.
\end{corollary}
\begin{proof}
By \cite[Theorem 4.4]{KP-amenable}, for each $X\in A\bimod A(\CBDF)$ every continuous derivation $A\to X'$ is inner. Now the result follows
from Theorem~\ref{topam-der}.
\end{proof}

If $A$ is a Banach algebra then the above notions coincide.

\begin{prop}
\label{am-top-am-banach}
Let $\sC\in\{\Fr,\CBDF\}$, and
let $A$ be a Banach algebra. Then $A$ is topologically amenable relative
to $\sC$ if and only if $A$ is amenable relative to $\Ban$.
\end{prop}
\begin{proof}
This follows immediately from Proposition \ref{banach-top-flat}.
\end{proof}

Since the algebras $k_0(V)$ that appear in the next section are not necessarily complete,
we adopt the following definition of topological amenability for non-complete algebras.

\begin{definition}
Let $\sC\in\{\Fr,\CBDF\}$, and let $A$ be a locally convex algebra with jointly
continuous multiplication such that $\wt{A}\in\alg(\sC)$
(where $\wt{A}$ is the completion of $A$).
We say that $A$ is {\em topologically amenable} relative to $\sC$
if $\wt{A}$ is topologically amenable relative to $\sC$.
\end{definition}

Given $A$ as above, let $A\bimod A(\sC)$ denote the category of locally convex
$A$-bimodules $X$ such that the left and right actions of $A$ on $X$ are jointly continuous
and such that the underlying space of $X$ is an object of $\sC$.
Clearly, we have an isomorphism of categories $A\bimod A(\sC)\cong\wt{A}\bimod\wt{A}(\sC)$.

Using the above definition, we can easily extend Theorem~\ref{topam-der} to non-complete
algebras.

\begin{theorem}
\label{topam-noncompl-der}
Let $\sC\in\{\Fr,\CBDF\}$, and let $A$ be a locally convex algebra with jointly
continuous multiplication such that $\wt{A}\in\alg(\sC)$.
Then $A$ is topologically amenable relative to $\sC$
if and only if for each $X\in A\bimod A(\sC)$ every continuous
derivation $A\to X'$ is inner.
\end{theorem}
\begin{proof}
Given $X\in A\bimod A(\sC)\cong\wt{A}\bimod\wt{A}(\sC)$, observe that $X'$ is complete
(see, e.g., \cite[28.5.(1)]{K1}). Hence each continuous derivation $A\to X'$ uniquely
extends to a continuous linear map $\wt{A}\to X'$, which is easily seen to be a derivation.
Thus we have a $1$-$1$ correspondence between the continuous derivations
$A\to X'$ and $\wt{A}\to X'$, which takes the inner derivations onto the inner derivations.
Now the result follows from Theorem~\ref{topam-der} applied to $\wt{A}$.
\end{proof}

We end this section with another consequence of topological amenability.
The proof is similar to that of \cite[Proposition 2.8.64]{D} therefore we omit it.

\begin{prop}
\label{top-am-dense-range}
Let $\sC\in\{\Fr,\CBDF\}$, and let $A$ and $B$ be locally convex algebras with jointly
continuous multiplication such that $\wt{A},\wt{B}\in\alg(\sC)$. Suppose that
$\theta\c A\to B$ is a continuous homomorphism with dense range.
If $A$ is topologically amenable relative to $\sC$, then so is $B$.
\end{prop}

\section{Topological Amenability for Co-echelon Algebras}
\label{sect:coechelon}

We are now going to investigate topological amenability in the framework of K\"othe co-echelon algebras.
Throughout this section, amenability and topological amenability are considered relative
to the category $\CBDF$ of complete barrelled DF-spaces.

The following result is a restatement of \cite[Lemma 0.5.1]{H2} adapted to
DF-spaces. The proof is essentially the same.

\begin{lemma}
\label{adjoint-surjective}
Let $X$ and $Y$ be DF-spaces such that $X$ is complete and $Y$ is quasi-barrelled,
and let $u\c X\to Y$ be a continuous linear injection. If $u$ has dense range and its adjoint $u'\c Y'\to X'$ is surjective, then $u$ is a topological isomorphism between $X$
and $Y$.
\end{lemma}
\begin{proof}
By assumption, $u'\c Y'\to X'$ is a continuous linear bijection between Fr\'echet spaces,
thus it is a topological isomorphism by the
Open Mapping Theorem \cite[Theorem 24.30]{MV}. Therefore $u''$ is
a topological isomorphism as well. We have
\begin{equation}
\label{bidual}
\iota_Y\circ u=u''\circ\iota_X,
\end{equation}
where $\iota_X\c X\hk X''$ and $\iota_Y\c Y\hk Y''$ are the canonical inclusions.
Since $Y$ is quasi-barrelled, it follows from \cite[11.2, Proposition 2]{J} that $\iota_Y$
is a topological embedding. Since $u''$ is a topological isomorphism,
we conclude from \eqref{bidual} that $\iota_X$ is continuous, or, equivalently,
a topological embedding [loc. cit.].
Hence $u''$ induces a topological isomorphism $u\c X\to\im u$.
Since $X$ is complete, $\im u$ is complete as well, so
$\im u$ is closed in $Y$. Therefore
$u$ is a topological isomorphism of $X$ onto $\im u=\overline{\im u}=Y$.
\end{proof}

Before proceeding to the characterization results, we list some properties of topologically amenable K\"othe co-echelon algebras of finite order.

\begin{lemma}
\label{ker_pi}
Let $1\le p<\infty$, and let $k_p(V)$ be a K\"othe co-echelon algebra.
Then the kernel of the multiplication map
$\pi\colon k_p(V)\Ptens k_p(V)\to k_p(V)$ is a complemented
subspace of $k_p(V)\Ptens k_p(V)$.
As a consequence, the quotient $k_p(V)\Ptens k_p(V)/\ker\pi$ is complete.
\end{lemma}
\begin{proof}
To begin with, let us show that the family $(e_i\otimes e_j)_{i,j\in\N}$
is a Schauder basis
in $k_p(V)\Ptens k_p(V)$ with respect to the square ordering of $\N\times\N$
(see \cite[Section 4.3]{RR}).
Indeed, we have $k_p(V)\Ptens k_p(V)=\ind_n \ell_p(v_n)\Ptens\ell_p(v_n)$
by \cite[Theorem 7]{EMM}.
Hence if $u\in k_p(V)\Ptens k_p(V)$ then
$u\in\ell_p(v_n)\Ptens\ell_p(v_n)$ for some $n\in\N$. Since $(e_j)_{j\in\N}$ is a Schauder basis in $\ell_p(v_n)$, it follows from \cite[Proposition 4.25]{RR}
that $(e_i\otimes e_j)_{i,j\in\N}$ is a Schauder basis
in $\ell_p(v_n)\Ptens\ell_p(v_n)$ with respect to the square ordering.
Therefore $u=\sum_{i,j=1}^{\infty}u_{ij}e_i\otimes e_j$
in $\ell_p(v_n)\Ptens\ell_p(v_n)$ hence also in $k_p(V)\Ptens k_p(V)$.
Consequently, $(e_i\otimes e_j)_{i,j\in\N}$ is a basis
in $k_p(V)\Ptens k_p(V)$. Since the coefficient functionals
$e_i^*\colon x\mapsto x_i$ on $k_p(V)$ are obviously continuous,
so are the functionals $e_i^*\otimes e_j^*$ on $k_p(V)\Ptens k_p(V)$.
Thus $(e_i\otimes e_j)_{i,j\in\N}$ is a Schauder basis.

Given $u=\sum_{i,j} u_{ij} e_i\otimes e_j\in k_p(V)\Ptens k_p(V)$,
we clearly have $\pi(u)=\sum_i u_{ii} e_i$. Hence
\[
\ker\pi=\overline{\spn}\{ e_i\otimes e_j : i\ne j\}.
\]
Therefore, to complete the proof, it suffices to construct a continuous
linear projection $P$ on $k_p(V)\Ptens k_p(V)$ such that
$P(e_i\otimes e_j)=\delta_{ij} e_i\otimes e_j$ for all $i,j$, where
$\delta_{ij}$ is the Kronecker delta.

Given $n\in\N$, let $\ell_p^0(v_n)$ denote the subspace of $\ell_p(v_n)$
consisting of finite sequences. Consider the bilinear map
\[
B_n\colon\ell_p^0(v_n)\times\ell_p^0(v_n)\to \ell_p(v_n)\Ptens\ell_p(v_n),\qquad
B_n(x,y)=\sum_{j=1}^\infty x_j y_j e_j\otimes e_j.
\]
We claim that $B_n$ is bounded. Indeed, using \cite[Lemma 2.22]{RR}, we obtain
\[
\sum_jx_jy_je_j\otimes e_j=\int_0^1\Big(\sum_jr_j(t)x_je_j\Big)\otimes\Big(\sum_jr_j(t)y_je_j\Big)dt,
\]
where $(r_j)$ are the Rademacher functions on $[0,1]$. Hence
\begin{align*}
\|B_n(x,y)\|_{\ell_p(v_n)\Ptens\ell_p(v_n)}
& \le\sup_{0\le t\le 1}\Big\|\sum_jr_j(t)x_je_j\Big\|_{\ell_p(v_n)}\Big\|\sum_jr_j(t)y_je_j\Big\|_{\ell_p(v_n)} \\
& =\|x\|_{\ell_p(v_n)}\|y\|_{\ell_p(v_n)}.
\end{align*}
Therefore $B_n$ is bounded. Extending $B_n$ by continuity to
$\ell_p(v_n)\times \ell_p(v_n)$ and then linearizing, we obtain
a bounded linear operator $P_n$ on $\ell_p(v_n)\Ptens\ell_p(v_n)$.
Finally, letting $P=\ind_n P_n$, we obtain a continuous linear
operator $P$ on $k_p(V)\Ptens k_p(V)$ with the required properties.
In view of the above remarks, this completes the proof.
\end{proof}

\begin{prop}
\label{middle-step}
Let $1\le p<\infty$ and let $k_p(V)$ be a K\"othe co-echelon algebra.
Suppose that $k_p(V)$ is topologically amenable. Then:
\begin{mycompactenum}
\item $V$ is eventually bounded;
\item the product map $\pi\c k_p(V)\pt k_p(V)\to k_p(V)$ is open, and there is a commutative diagram
\[\begin{tikzcd}
k_p(V)\pt k_p(V)\arrow[swap]{d}{\pi}\arrow{r}{q} & k_p(V)\pt k_p(V)/\ker\pi\arrow[shift left]{dl}{\hat{\pi}} \\
k_p(V)\arrow{ur}{\hat{\pi}^{-1}} &
\end{tikzcd}\]
where $q$ is the quotient map. Moreover,
\begin{equation}
\hat{\pi}^{-1}(a)=\sum_{j=1}^{\infty}a_je_j\otimes e_j+\ker\pi\hspace{25pt}(a\in k_p(V));
\label{inverse-formula}
\end{equation}
\item $k_p(V)$ is nuclear.
\end{mycompactenum}
\end{prop}
\begin{proof}
(i) Suppose towards a contradiction that all the weights $v_n$
are unbounded.
This implies that there is a sequence $j_l\nearrow\infty$ such that
$v_k(j_l)\ge1$ for all $l\in\N$ and all $k\le l$. Define a dense range homomorphism
\[
\theta\c k_p(V)\to\ell_p,\quad\theta(a):=(a_{j_l})_{l\in\N},
\]
where we consider $\ell_p$ with the coordinate-wise multiplication. For every $k\in\N$ we get
\[
\|\theta(a)\|_{\ell_p}^p=\sum_{l\le k}|a_{j_l}|^p+\sum_{l>k}|a_{j_l}|^p\le C_k\|a\|_{k,p}^p
\]
with $C_k:=\max\{1/v_k(j_l)^p\c\,l\le k\}+1$. Consequently, $\theta$ indeed takes
$k_p(V)$ to $\ell_p$ and is continuous.
Since $k_p(V)$ is topologically amenable, it follows from
Propositions~\ref{am-top-am-banach} and \ref{top-am-dense-range} that the Banach algebra $\ell_p$ is amenable. This leads to a contradiction since $\ell_p$ is known to be
non-amenable (see, e.g., \cite[Example 4.1.42(iii)]{D}).
Therefore $V$ is eventually bounded.

(ii) To prove that $\pi$ is open, it suffices to show that $\hat\pi$ is a topological isomorphism.
Taking into account Lemma~\ref{ker_pi}, we see that $\hat\pi$ acts between complete
barrelled DF-spaces and, clearly, has dense range. By Lemma~\ref{adjoint-surjective},
the proof will be complete if we show that $\hat\pi'$ is surjective. Towards this goal,
take $\psi\in (k_p(V)\Ptens k_p(V)/\ker\pi)'$ and let $\psi_0=\psi\circ q$.
Since $\psi_0$ vanishes on $\ker\pi$, we have
\begin{equation}
\psi_0(a\otimes b)=\sum_{j=1}^{\infty}a_jb_j\psi_0(e_j\otimes e_j)\hspace{25pt}(a,b\in k_p(V)).
\label{derivation}
\end{equation}
Define now a linear map
\[\delta\c k_p(V)\to(k_p(V)\otimes\CC)',\qquad\langle b,\delta(a)\rangle:=\psi_0(a\otimes b).\]
In other words, $\delta$ is the image of $\psi_0$ under \eqref{adjass} (where $X=Y=k_p(V)$
and $Z=\CC$). Hence $\delta$ is continuous.
Using \eqref{derivation}, we see that
\[\la c,\delta(ab)\ra=\la ab\otimes c,\psi_0\ra
=\la a\otimes bc,\psi_0\ra=\la c,\delta(a)\cdot b\ra\hspace{25pt}(a,b,c\in k_p(V)).\]
Since the left action of $k_p(V)$ on $(k_p(V)\otimes\CC)'$ is trivial, we conclude that
$\delta$ is a derivation. By Theorem \ref{topam-der}, there is $\phi\in(k_p(V))'$ such that
\[\delta(a)=\phi\cdot a\hspace{25pt}(a\in k_p(V)).\]
Hence for all $a,b\in k_p(V)$ we have
\[
\la a\otimes b+\ker\pi,\hat{\pi}'(\phi)\ra
=\la ab,\phi\ra
=\la b,\phi\cdot a\ra
=\la b,\delta(a)\ra
=\la a\otimes b,\psi_0\ra
=\la a\otimes b+\ker\pi,\psi\ra,
\]
that is, $\hat{\pi}'(\phi)=\psi$. Therefore the map $\hat{\pi}'$ is surjective.
In view of the above remarks, this implies that $\pi$ is open.
To prove \eqref{inverse-formula}, observe that for every $j\in\N$ we have
\[\hat{\pi}^{-1}(e_j)=\hat{\pi}^{-1}\circ\hat{\pi}(e_j\otimes e_j+\ker\pi)=e_j\otimes e_j+\ker\pi.\]
Since $(e_j)_{j\in\N}$ is a Schauder basis in $k_p(V)$, this implies \eqref{inverse-formula}.

(iii) To get the nuclearity of $k_p(V)$ we repeat exactly the proof of \cite[Theorem 5.1]{KP-amenable}.
We can indeed do so, since $\hat{\pi}^{-1}$ is a topological isomorphism not only in the case of amenability (which was the assumption in \cite{KP-amenable}) but also under the weaker assumption of topological amenability.
\end{proof}

\begin{theorem}
\label{finite-order}
Let $1\le p<\infty$, and let $k_p(V)$ be a K\"othe co-echelon algebra.
TFAE:
\begin{mycompactenum}
\item $k_p(V)$ is topologically amenable;
\item $k_p(V)$ is amenable;
\item $k_p(V)$ is contractible;
\item $k_p(V)$ is unital;
\item $V$ is eventually in $\ell_1$;
\item $V$ is eventually bounded, and $k_p(V)$ is nuclear.
\end{mycompactenum}
\end{theorem}

\begin{proof}
$\mathrm{(ii)\Leftrightarrow (iii)\Leftrightarrow (iv)}$: see \cite[Theorem 5.1]{KP-amenable}.

$\mathrm{(iv)\Leftrightarrow (v)\Leftrightarrow (vi)}$: see \cite[Proposition 2.5]{BonDom}.

$\mathrm{(ii)\Rightarrow (i)}$ follows from Corollary \ref{am-top-am}.

$\mathrm{(i)\Rightarrow (vi)}$ follows from Proposition \ref{middle-step}.
\end{proof}

It turns out that the cases of K\"othe co-echelon algebras of order zero and infinity can be treated simultaneously.

\begin{theorem}
\label{thm:kinf}
Let $p\in\{0,\infty\}$, and let $k_p(V)$ be a K\"othe co-echelon algebra. TFAE:
\begin{mycompactenum}
\item $k_0(V)$ is topologically amenable;
\item $k_{\infty}(V)$ is topologically amenable;
\item $V$ is eventually bounded.
\end{mycompactenum}
\end{theorem}

\begin{proof}
$\mathrm{(ii)\Rightarrow (iii)}$. If $k_{\infty}(V)$ is topologically amenable then we can follow the proof of Proposition \ref{middle-step} to show that $V$ is eventually bounded.
Indeed,
suppose towards a contradiction that all the weights $v_n$
are unbounded.
This implies that there is a sequence $j_l\nearrow\infty$ such that
$v_k(j_l)\ge 2^l$ for all $l\in\N$ and all $k\le l$. Define a dense range homomorphism
\[
\theta\c k_\infty(V)\to\ell_1,\quad\theta(a):=(a_{j_l})_{l\in\N},
\]
where we consider $\ell_1$ with the coordinate-wise multiplication. For every $k\in\N$ we get
\[
\|\theta(a)\|_{\ell_1}=\sum_{l\le k}|a_{j_l}|+\sum_{l>k}|a_{j_l}|\le C_k\|a\|_{k,\infty}
\]
with $C_k:=\sum_{l\le k} (1/v_k(j_l))+1$. Consequently, $\theta$ indeed takes
$k_\infty(V)$ to $\ell_1$ and is continuous.
Since $k_\infty(V)$ is topologically amenable, it follows from
Propositions~\ref{am-top-am-banach} and \ref{top-am-dense-range} that the Banach algebra $\ell_1$ is amenable. This leads to a contradiction since $\ell_1$ is known to be
non-amenable (see, e.g., \cite[Example 4.1.42(iii)]{D}).
Therefore $V$ is eventually bounded.

$\mathrm{(iii)\Rightarrow (ii)}$.
Without loss of generality, we may assume that $v_1\in V$ is bounded.
We then have $\ell_\infty\subset\ell_\infty(v_1)$, and the inclusion is clearly bounded.
Composing with the inclusion of $\ell_\infty(v_1)$ into $k_\infty(V)$, we obtain
a continuous homomorphism
\begin{equation}
\theta\c\ell_{\infty}\to k_{\infty}(V),\quad \theta(a):=a.
\label{infinite-order-density}
\end{equation}
We claim that $\theta$ has dense range. To this end, let $a\in k_{\infty}(V)\setminus\{ 0\}$, i.e.,
$0<\|a\|_{n,\infty}<\infty$ for some $n\in\N$. Using (W3),
find $m\in\N$ and $C>0$ such that
\[
\forall\,j\in\N\quad v_m(j)\le C v_n(j)^2.
\]
Fix $\ve>0$ and denote $J_1:=\{j\in\N\c\,v_n(j)\ge\frac{\ve}{2C\|a\|_{n,\infty}}\}$ and $J_2:=\N\setminus J_1$. Define a scalar sequence $b^{\ve}=(b_j)_j$ as
\[b_j:=\begin{cases}
a_j,\,\,j\in J_1 \\
0,\,\,\,j\in J_2.
\end{cases}\]
For each $j\in J_1$ we have
\[|b_j|\le\frac{2C}{\ve}\|a\|_{n,\infty}|a_j|v_n(j)\le\frac{2C}{\ve}\|a\|_{n,\infty}^2.\]
Consequently, $b^{\ve}\in\ell_{\infty}$ with $\|b^{\ve}\|_{\ell_{\infty}}\le \frac{2C}{\ve}\|a\|_{n,\infty}^2$.
If $J_2$ is empty, we conclude that $a=b^\ve$ is in the range of $\theta$.
Otherwise observe that
\[\|a-b^{\ve}\|_{m,\infty}=\sup_{j\in J_2}|a_j|v_m(j).\]
For any $j\in J_2$ we get
\begin{align*}
|a_j|v_m(j)\le C|a_j|v_n(j)^2 & \le C\|a\|_{n,\infty}v_n(j) \\
& <C\|a\|_{n,\infty}\frac{\ve}{2C\|a\|_{n,\infty}}=\frac{\ve}{2}<\ve.
\end{align*}
Thus $\|a-b^{\ve}\|_{m,\infty}<\ve$. This implies that for a sequence $\ve_k\searrow0$ we get another sequence $b^k:=b^{\ve_k}\in\ell_{\infty}$ such that
\[\lim_{k\to\infty}b^k=a\quad\text{in}\,\,\ell_{\infty}(v_m).\]
But the topology of $\ell_{\infty}(v_m)$ is stronger than that of $k_{\infty}(V)$, thus
\[\lim_{k\to\infty}b^k=a\quad\text{in}\,\,k_{\infty}(V).\]
Consequently, the homomorphism \eqref{infinite-order-density} has dense range.
Since $\ell_\infty$ is amenable by \cite[Lemma 7.10]{BEJ} (see also \cite[Theorem 5.6.2]{D},
\cite[Theorem VII.2.42]{H2}), the topological amenability of $k_\infty(V)$
now follows from Propositions \ref{am-top-am-banach} and \ref{top-am-dense-range}.

$\mathrm{(i)\Leftrightarrow(iii)}$.
This part is even easier since $(e_j)_{j\in\N}$ is a common Schauder basis for both
$c_0$ and $k_0(V)$, thus the density of the range of $\theta$
in \eqref{infinite-order-density} is immediate.
\end{proof}

\section{Examples}
\label{sect:examples}

Let us now give some concrete examples which illustrate
Theorems~\ref{finite-order} and~\ref{thm:kinf}.

\begin{example}
\label{example:phi_amen}
Applying Theorem \ref{finite-order}, we see that the algebra $\phi$ of finite
sequences (see Example~\ref{example:phi}) is not topologically amenable.
\end{example}

\begin{example}
\label{example:dual_power_amen}
Consider the dual power series space $D\Lambda_R^p(\alpha)$, where $1\le p\le\infty$
and $R\in\{ 0\}\cup [1,+\infty)$ (see Example~\ref{example:dual_power}).
If $R\ge 1$, then the respective weights $(r^{\alpha_j})_{j\in\N}$
are clearly unbounded for all $r>R$, so $D\Lambda_R^p(\alpha)$ is not topologically amenable
in this case (see Theorems~\ref{finite-order} and~\ref{thm:kinf}).
On the other hand, $(r^{\alpha_j})_{j\in\N}$ is bounded for each $0<r\le 1$,
and so $D\Lambda_0^\infty(\alpha)$ is topologically amenable by Theorem~\ref{thm:kinf}.

In fact, more is true. Indeed, all dual power series spaces $D\Lambda_R^p(\alpha)$
are Schwartz spaces by \cite[Theorem 4.9]{BMS}.
Also, it is clear that $D\Lambda_0^\infty(\alpha)$ is unital.
Now \cite[Theorem 12]{KP-contractible} implies that $D\Lambda_0^\infty(\alpha)$ is contractible.

Finally, if $p<\infty$, then $D\Lambda_0^p(\alpha)$ is topologically amenable iff it is
contractible iff $\sum_j r^{\alpha_j}<\infty$ for some $r>0$ (see Theorem~\ref{finite-order}).
\end{example}

\begin{example}
The algebra $s'$ of sequences of polynomial growth is contractible.
This follows from \cite[Proposition 7.3]{JLT_fmwk} and is explicitly
mentioned in \cite[Example 3.1]{Pir_Oulu}, \cite[Example 6.6]{ZL}.
Since $s'=D\Lambda_0^p(\alpha)$, where $\alpha_j=\log j$
(see Example~\ref{example:s'}), we see that the contractibility of $s'$
is also a special case of Example~\ref{example:dual_power_amen}.
\end{example}

\begin{example}
\label{example:germs_amen}
As another special case of Example~\ref{example:dual_power_amen}, we see that
the Hadamard algebra $\cH(\ol{\DD}_R)$ of germs of holomorphic
functions on the disc $\ol{\DD}_R$
(see Example~\ref{example:germs}) is not topologically amenable for $R\ge 1$.
On the other hand, letting $R=0$, we see that the Hadamard algebra
$\cH_0$ of holomorphic germs at zero is contractible.
\end{example}

The reader may have noticed that for all the algebras mentioned in
Examples~\ref{example:phi_amen}--\ref{example:germs_amen}
topological amenability is equivalent to contractibility.
On the other hand, there are two obvious examples of topologically amenable co-echelon
algebras that are not contractible --- namely, $c_0$ and $\ell_\infty$.
To construct more examples of the same kind, let us first observe that the direct
sum of two co-echelon algebras of the same order is also a co-echelon algebra.
More exactly, if $V=(v_n)_{n\in\N}$ and $W=(w_n)_{n\in\N}$ are sequences of weights
on index sets $I$ and $J$, respectively, then we have
$k_p(I,V)\mathop{\oplus} k_p(J,W)\cong k_p(I\sqcup J,U)$,
where the sequence $U=(u_n)_{n\in\N}$ of weights on $I\sqcup J$ is given by
$u_n(i)=v_n(i)$ if $i\in I$, and $u_n(j)=w_n(j)$ if $j\in J$.
Conversely, each partition $I=S\sqcup T$ induces a direct
sum decomposition $k_p(I,V)\cong k_p(S,V_S)\mathop{\oplus} k_p(T,V_T)$,
where $V_S$ and $V_T$ consist of the restrictions to $S$ and $T$ of weights from $V$.

\begin{example}
\label{example:dirsum}
Let $A_1=c_0\mathop{\oplus} D\Lambda_0^\infty(\alpha)$
and $A_2=\ell_\infty\mathop{\oplus} D\Lambda_0^\infty(\alpha)$.
In view of the above discussion,
$A_1$ and $A_2$ are co-echelon algebras of order $0$ and $\infty$, respectively.
By Theorem~\ref{thm:kinf}, $A_1$ and $A_2$ are topologically amenable.
On the other hand, $A_1$ and $A_2$ are not Montel spaces, so they are not
contractible by \cite[Theorems 12 and 13]{KP-contractible}
(moreover, $A_1$ is not unital, which already implies that it is not contractible).
\end{example}

Of course, the above example is degenerate in a sense. Our next goal is to
construct a ``genuine'' example of a co-echelon algebra of order $\infty$ which is
topologically amenable and unital, but is not contractible.
By ``genuine'' we mean that the algebra we are going to construct
is not reduced to a direct sum of $\ell_\infty$ with a contractible
algebra of the form $k_\infty(V)$ in the sense explained before Example~\ref{example:dirsum}.

\begin{example}
\label{example:NN}
We fix a sequence $(c_j)_{j\in\N}$ of positive numbers such that $c_j\le 1$ for all $j$,
and such that $c_j\to 0$ as $j\to\infty$. For each $n\in\N$ define a weight $v_n$ on $\N^2$ by
\begin{equation}
\label{v_ij}
v_n(i,j)=
\begin{cases}
c_j^n, & i<n,\\
1, & i\ge n.
\end{cases}
\end{equation}
Clearly, the sequence $V=(v_n)_{n\in\N}$ satisfies (W1) and (W2). Furthermore, we have
$v_{2n}\le v_n^2$ for all $n\in\N$, whence $V$ satisfies (W3).
Thus $k_p(\N^2,V)$ is a K\"othe co-echelon algebra for all $p$.
Since $V$ is eventually bounded, we see that $k_0(\N^2,V)$ and
$k_\infty(\N^2,V)$ are topologically amenable (see Theorem~\ref{thm:kinf}).
Moreover, $k_\infty(\N^2,V)$ is clearly unital.
\end{example}

For each $i\in\N$, let $L_i=\{ (i,j) : j\in\N\}\subset\N^2$.

\begin{lemma}
\label{lemma:fin_inter}
If $S\subset\N^2$, then $k_\infty(S,V_S)$ is a Banach space if and only if
$S\cap L_n$ is finite for all $n\in\N$.
\end{lemma}
\begin{proof}
We will use the well-known fact that an (LB)-space
$E=\ind_n E_n$ (where $E_n$ are Banach spaces, and $E_n\to E_{n+1}$
are bounded linear injections) is a Banach space if and only if the sequence $(E_n)$
stabilizes in the sense that there exists $N\in\N$ such that $E_n\to E_{n+1}$
is a topological isomorphism for all $n\ge N$ (this follows, for example, from
\cite[19.5.(4)]{K1}).

If $S\cap L_n$ is finite for all $n\in\N$, then so is
$S_n=\bigcup_{k\le n} (S\cap L_k)$.
We clearly have $v_n=v_{n+1}=1$ outside $S_n$. Letting
\[
C_n=\max_{(i,j)\in S_n} \frac{v_n(i,j)}{v_{n+1}(i,j)},
\]
we obtain
$v_n\le C_n v_{n+1}$ everywhere on $S$.
This readily implies that $\ell_\infty(S,v_n)\to\ell_\infty(S,v_{n+1})$ is a topological
isomorphism. Hence $k_\infty(S,V_S)$ is a Banach space.

Conversely, suppose that $S\cap L_k$ is infinite for some $k$.
Since for each $n\ge k+1$ we have $v_n(k,j)=c_j^n$, and since $c_j\to 0$ as $j\to\infty$,
we see that there is no $C>0$ such that $v_n\le C v_{n+1}$ on $S\cap L_k$.
Therefore $\ell_\infty(S,v_n)\to\ell_\infty(S,v_{n+1})$ is not a topological isomorphism,
and so $k_\infty(S,V_S)$ is not a Banach space.
\end{proof}

\begin{lemma}
\label{lemma:no_decomp}
There is no decomposition $\N^2=S\sqcup T$ such that $k_\infty(S,V_S)$
is a Banach space and such that $k_\infty(T,V_T)$ is a Montel space.
\end{lemma}
\begin{proof}
Suppose that $S\subset\N^2$ is a subset such that $k_\infty(S,V_S)$ is a Banach space,
and let $T=\N^2\setminus S$. By Lemma~\ref{lemma:fin_inter},
for each $n\in\N$ there exists $j_n\in\N$ such that $(n,j_n)\in T$.
We clearly have $v_m(n,j_n)=1$ for all $n\ge m$. Letting $R=\{ (n,j_n) : n\in\N\}$,
we conclude that
\[
\inf_{(i,j)\in R}\frac{v_m(i,j)}{v_1(i,j)}
=\inf_{n\in\N} \frac{v_m(n,j_n)}{v_1(n,j_n)}
>0 \qquad (m\in\N).
\]
The existence of an infinite set $R\subset T$ with the above property
means precisely that $k_\infty(T,V_T)$ is not Montel
\cite[Theorem 4.7]{BMS}.
\end{proof}

Essentially the same argument applies to $k_0(\N^2,V)$. However, more is true.

\begin{lemma}
\label{lemma:no_decomp_2}
Let $V$ be the weight sequence on $\N^2$ given by
\eqref{v_ij}. Then $k_0({\N^2},V)$ is not complete,
and the underlying lcs of $k_0(\N^2,V)$
is not isomorphic to a direct sum of a normed space and a dense subspace of a
reflexive space.
\end{lemma}
\begin{proof}
Recall from \cite[Theorem 2.7]{BMS} that, for each set $I$ and each sequence $V=(v_n)_{n\in\N}$
of weights on $I$ satisfying (W1) and (W2), the strong dual of $k_0(I,V)$ is
topologically isomorphic to the K\"othe echelon space
\[
\lambda_1(I,A)=\Bigl\{ x=(x_i)\in\CC^I : \| x\|_n=\sum_{i\in I} |x_i| a_n(i)<\infty\;\forall n\in\N\Bigr\},
\]
where $a_n(i)=v_n(i)^{-1}$ and $A=(a_n)_{n\in\N}$ is the corresponding K\"othe set.

Let now $I=\N^2$, let $V$ be given by \eqref{v_ij}, and let $E=k_0(\N^2,V)$.
Assume, towards a contradiction, that
$E\cong E_0\mathop{\oplus} E_1$, where $E_0$ is a normed space and $E_1$ is
a dense subspace of a reflexive space.
Hence we have a topological isomorphism $E'\cong E'_0\mathop{\oplus} E'_1$.
Moreover, $E'_0$ is a Banach space, and $E'_1$ is a reflexive Fr\'echet space
(see, e.g., \cite[23.5.(5) and 29.3.(1)]{K1}).
Now recall from \cite[Corollary 25.14]{MV} that all reflexive Fr\'echet spaces
are distinguished, i.e., their strong duals are barrelled.
Clearly, each normed space is distinguished, and a direct sum of two distinguished
spaces is distinguished. Therefore $E'$ is distinguished.

On the other hand, it is easily seen that the K\"othe set $A=(a_n)_{n\in\N}$ on $\N^2$,
where $a_n(i,j)=v_n(i,j)^{-1}$, satisfies the conditions of \cite[Corollary 27.18]{MV}.
Hence $\lambda_1(\N^2,A)$ is not distinguished.
This is a contradiction since $E'\cong\lambda_1(\N^2,A)$ (see above).

Applying now \cite[Corollary 3.5 and Theorem 3.7]{BMS}, we conclude that
$k_0(\N^2,V)$ is not complete.
\end{proof}

\begin{prop}
Let $V$ be the weight sequence on $\N^2$ given by
\eqref{v_ij}. Then
\begin{mycompactenum}
\item
$k_0(\N^2,V)$ and $k_\infty(\N^2,V)$ are topologically
amenable K\"othe co-echelon algebras;
\item
$k_\infty(\N^2,V)$ is unital;
\item
$k_0(\N^2,V)$ is not complete;
\item
there is no decomposition $\N^2=S\sqcup T$ such that $k_\infty(S,V_S)$
is a Banach algebra and such that $k_\infty(T,V_T)$ is a contractible algebra;
\item
the underlying lcs of $k_0(\N^2,V)$
is not isomorphic to a direct sum of a normed algebra and a contractible
K\"othe co-echelon algebra.
\end{mycompactenum}
\end{prop}
\begin{proof}
Properties (i) and (ii) are mentioned in Example~\ref{example:NN}, while (iii)
is contained in Lemma~\ref{lemma:no_decomp_2}.
To prove (iv) and (v), observe that each contractible co-echelon algebra of order $p>0$
is a Montel space (for $p<\infty$ this follows from Theorem~\ref{finite-order},
while for $p=\infty$ this is \cite[Theorem 12]{KP-contractible}). Also, if a co-echelon
algebra of order $0$
is contractible, then its completion is a Montel space \cite[Theorem 13]{KP-contractible}.
Now (iv) and (v) follow from Lemmas~\ref{lemma:no_decomp}
and~\ref{lemma:no_decomp_2}, respectively.
\end{proof}

We conjecture that (v) holds for $k_\infty(\N^2,V)$ as well.

\end{document}